\documentclass[12pt,draft]{amsart}
\usepackage[all]{xy}
\usepackage{amsfonts}
\usepackage{amssymb}

\usepackage[english]{babel}

\newtheorem{lemma}{Lemma}
\newtheorem{theorem}{Theorem}

\theoremstyle{definition}
\newtheorem{definition}{Definition}
\newtheorem{remark}{Remark}

\begin{document}

\title{Reidemeister classes, wreath products and solvability}

\author{Evgenij Troitsky}
\thanks{The work is supported by the Russian Science Foundation under grant 21-11-00080.}
\address{Moscow Center for Fundamental and Applied Mathematics, MSU Department,\newline
	Dept. of Mech. and Math., Lomonosov Moscow State University, 119991 Moscow, Russia}
\email{troitsky@mech.math.msu.su}

\begin{abstract}
Reidemeister (or twisted conjugacy) classes are considered in
restricted wreath products of the form $G\wr \mathbb{Z}^k$, where $G$ is a finite group.

For an automorphism $\varphi$ of finite order (supposed to be the same for the torsion subgroup $\oplus G$ and the quotient $\mathbb{Z}^k$) with finite number $R(\varphi)$ of  Reidemeister classes,
this number is identified with the number of equivalence classes of finite-dimensional unitary irreducible  representations  of the product that are fixed by the dual homeomorphism $\widehat{\varphi}$ (i.e. the so-called conjecture TBFT$_f$ is proved in this case).  

For these groups and automorphisms, we prove the following conjecture: if a finitely generated residually finite group has an automorphism with $R(\varphi)<\infty$ then it is solvable-by-finite (so-called conjecture R). 
\end{abstract}

\maketitle

\emph{Reidemeister}, or \emph{twisted conjugacy classes} of an automorphism $\varphi$ of a group $\Gamma$ are equivalence classes with respect to $x\sim yx\varphi(y^{-1})$. Their number
$R(\varphi)$ (finite or infinite) is called \emph{Reidemeister number}. 

The following three directions form the mainstream of the current study of Reidemeister classes:
\begin{enumerate}
\item[1)] To prove or disprove the so-called TBFT (twisted Burnside-Frobenius theory) conjecture:
the Reidemeister number $R(\varphi)$ (if finite) coincides with the number 
of equivalence classes of irreducible unitary representations of $\Gamma$ fixed by 
the induced homeomorphism $\widehat{\varphi}$ of
the unitary dual $\widehat{\Gamma}$. Its finite-dimensional version TBFT$_f$ was also studied (here one considers only finite-dimensional fixed representations).
Also this property can be considered for an individual automorphism.
The most important classes of groups for which TBFT$_f$ is true, are polycyclic-by-finite groups 
\cite{polyc} and residually finite groups of finite Pr\"ufer rank \cite{TroitskyFiniteRank2022}.
On the other hand, in \cite{FeTrZi20} we have detected an example of infinitely generated
residually finite group which has neither TBFT nor TBFT$_f$.  

\item[2)] As an opposite case, to determine classes of groups for which any automorphism has 
infinite Reidemeister number (this property is called $R_\infty$) --- the list of results here is very
extended, we mention only some expository or recent papers: \cite{FelLeonTro,FelTroJGT,FraimanTroitsky,
TroitskyFiniteRank2022,MitraSankaran2022,Tertooy2022}.

\item[3)] To study rationality and other properties of Reidemeister zeta function constructed
from $R(\varphi^n)$ (see e.g. \cite{FelTroNilpFR2021RJMP} for a recent progress).  
\end{enumerate}

We will deal with the first two aspects in this paper. 
Inspired by \cite{Jabara} we have formulated in \cite{FelTroJGT} the following conjecture:

\medskip
\textbf{Conjecture R.} \emph{Let $\Gamma$ be a finitely generated residually finite group. Either $\Gamma$
is $R_\infty$, or $\Gamma$ is solvable-by-finite.}

\medskip
The conjecture was discussed in several papers, in particular in \cite{GoncalvesNasybullov2019CA}.  
The conjecture was supported by the following recent result \cite{TroitskyFiniteRank2022}:
any residually finite group of finite upper rank admitting an automorphism $\varphi$ with finite Reidemeister number $R(\varphi)$ is solvable-by-finite. 

The main results of the present paper are:

\smallskip
\textbf{A)} TBFT$_f$ is true for automorphisms of finite order of a restricted wreath product
$G\wr \mathbb{Z}^k$, where $G$ is an arbitrary finite group
(Theorem \ref{teo:main_TBFT}). The proof needs a technical restriction (\ref{eq:condfinord}),
which is typically fulfilled automatically (see e.g. \cite{TroLamp,troitskii2023reidemeister550080588}).

\smallskip
\textbf{B)} The above conjecture (R) is true in this situation. More specifically, suppose that 
a restricted wreath product $\Gamma=G\wr \mathbb{Z}^k$, where $G$ is an arbitrary finite group,
admits an automorphism $\varphi$ of finite order with $R(\varphi)<\infty$. Then (under the same restriction)
$\Gamma$ is solvable-by-finite (in fact, solvable) (Theorem \ref{teo:ja_wr}).

The statement A) generalizes a part of results of \cite{TroLamp,troitskii2023reidemeister550080588}.

\medskip
Recall some facts to be used in the proofs and give some modifications of them.

Let $\mathbf{F}(\varphi):=\{g\in \Gamma \colon \varphi(g)=g\}$ be the subgroup  of $\varphi$-fixed elements.
We use the notation $\tau_g(x)=gxg^{-1}$ for an inner automorphism and for its restriction to a normal
subgroup.

From the equality
$
yg\varphi(y^{-1})x= ygx x^{-1}\varphi (y^{-1})x= y(gx) (\tau_{x^{-1}} \circ \varphi )(y^{-1}),
$
it follows a very useful statement (see e.g. \cite{polyc}):
\begin{lemma}\label{lem:shift}
Shifts of Reidemeister classes of $\varphi$ are Reidemeister classes of $\tau_{x^{-1}} \circ \varphi$:
$$
\{g\}_\varphi x= \{gx\}_{\tau_{x^{-1}} \circ \varphi}.
$$	
Hence, $R(\tau_g\circ \varphi)=R(\varphi)$.
\end{lemma}

\begin{lemma}\label{lem:Jab_fin_ord}
\emph{(\cite{Jabara})}
Suppose that $\Gamma$ is a residually finite group and $\varphi:\Gamma\to\Gamma$ is an automorphism of finite order with $R(\varphi)<\infty$. Then $|\mathbf{F}({\varphi})|< \infty$.
\end{lemma}

\begin{lemma}\label{lem:Jab_fin}
\emph{(Prop. 3.4 in \cite{FelLuchTro})}
Suppose that $\Gamma$ is a finitely generated residually finite group and $\varphi:\Gamma\to\Gamma$ is an automorphism with $R(\varphi)<\infty$. Then $|\mathbf{F}({\varphi})|< \infty$.
\end{lemma}

\begin{lemma}\label{lem:extensions}
\emph{(\cite{FelHill,go:nil1}, see also \cite{polyc})}
Suppose, $\varphi:\Gamma\to \Gamma$ is an automorphism, $H$ is a normal $\varphi$-invariant subgroup of $\Gamma$, so  $\varphi$ induces automorphisms $\varphi':H\to H$ and $\widetilde{\varphi}:\Gamma/H \to \Gamma/H$.
Then
	\begin{itemize}
		\item the projection $\Gamma\to \Gamma/H$ maps Reidemeister classes of $\varphi$ onto Reidemeister classes of $\widetilde{\varphi}$, in particular $R(\widetilde{\varphi})\le R(\varphi)$;

		\item if $|\mathbf{F}(\widetilde{\varphi})|=n$, then $R(\varphi')\le R(\varphi)\cdot n$;

		\item if $\mathbf{F}(\widetilde{\varphi})=\{e\}$, then 		each Reidemeister class of $\varphi'$ is an intersection of the appropriate Reidemeister class of $\varphi$ and $H$; 

		\item if $\mathbf{F}(\widetilde{\varphi})=\{e\}$, then 
		$R(\varphi)=\sum_{j=1}^R R(\tau_{g_j} \circ \varphi')$, where $g_1,\dots g_R$ are some elements of $\Gamma$ such that
		$p(g_1),\dots,p(g_R)$ are representatives of all Reidemeister classes of $\widetilde{\varphi}$, 
$p:\Gamma\to \Gamma/H$ is the natural projection and $R=R(\widetilde{\varphi})$.
	\end{itemize}	
\end{lemma}

The following statement was obtained in \cite{Rowley1995} using the classification of finite simple groups. 
\begin{theorem}\label{teo:Rowley95}
A finite group with a fixed-point-free (i.e., regular) automorphism is solvable.
\end{theorem}
 
Recall also the following folklore observation.
 
\begin{lemma}\label{lem:fix_fin}
For an automorphism $f:F\to F$ of a finite group, $R(f)>1$ if and only if $\mathbf{F}(f)\ne \{e\}$.
\end{lemma}

\begin{proof}
Indeed, consider the Reidemeister class $\{e\}_f$ as an orbit of the twisted action of $G$ on itself:
$g: x \mapsto gx f(g^{-1})$. Then by the orbit-stabilizer theorem, $R(f)=1$, i.e. $\{e\}_f=G$, if and only if the stabilizer of $e$ under the twisted action is trivial. But $x e f(x^{-1}) =e$ if and only if $x\in \mathbf{F}(f)$.
\end{proof}

We say that Reidemeister classes of ${\varphi}:\Gamma \to \Gamma$  \emph{are separated by an epimorphism} $f: \Gamma \to F$ onto a finite 
group $F$ if $f$ induces a bijection of Reidemeister classes.

The following statement can be extracted from \cite{polyc,FelLuchTro}.
\begin{lemma}\label{lem:tbft_separ}
Let ${\varphi}:\Gamma \to \Gamma$ have $R(\varphi)<\infty$.
Then TBFT$_f$ is true for $\varphi$ if and only if Reidemeister classes of ${\varphi}$  are separated by an epimorphism $f: \Gamma \to F$ onto a finite 
group $F$.
\end{lemma}

\begin{proof}
Indeed, each equivalence class of irreducible unitary finite-dimensional representation $[\rho]$,
that is $\widehat{\varphi}$-fixed,
 gives rise to a function
$g\mapsto \mathrm{Trace} (V \rho(g))$, where $V$ is the unitary intertwining operator between
$\rho$ and $\rho\circ\varphi$.
This function is constant on Reidemeister classes. 
It is a matrix coefficient and these functions are the only matrix coefficients that are constant
on Reidemeister classes. To see this it is sufficient to represent an arbitrary matrix coefficient
as a twisted trace and obtain that the twist is given by an intertwining operator.  
For distinct representation classes these matrix coefficients are linearly independent. 
Thus, the number of twisted conjugacy classes $R(\varphi)$ is more or equal to the number of $\widehat{\varphi}$-fixed unitary irreducible finite-dimensional representations $\mathrm{Fix}(\widehat{\varphi})$
(here we denote by $\widehat{\varphi}$ the induced homeomorphism of the finite-dimensional part of the unitary dual).

If Reidemeister classes are separated by $f$, then $R(\varphi)=R(\varphi_F)$, where  
$\varphi_F:F\to F$ is the induced automorphism. Then 
$
\mathrm{Fix}(\widehat{\varphi}) \le R(\varphi)=R(\varphi_F) =
\mathrm{Fix}(\widehat{\varphi_F}) \le \mathrm{Fix}(\widehat{\varphi}). 
$

Conversely, suppose that $R(\varphi)=|\mathrm{Fix}(\widehat{\varphi})|$   and $\mathrm{Fix}(\widehat{\varphi})=
\{\rho_1,\dots,\rho_r\}$.
Since $R(\varphi)$ is finite, each $\rho_i$ is finite, i.e. factorizes through $f_i: \Gamma \to F_i$,
$\rho_i=\widetilde{\rho}_i \circ f_i$, for some finite group $F_i$.
Then the map $f=f_1\oplus\cdots \oplus f_r; \Gamma \to F_1\oplus\cdots \oplus F_r=F$
separates Reidemeister classes because, for natural projections $p_i :F\to F_i$, the representations
$ \widetilde{\rho}_i \circ p_i$ of $F$ are fixed and non-equivalent. Thus
$
R( \varphi)\ge R(\varphi_F)\ge r = R( \varphi),
$
i.e., $f$ separates classes. This completes the proof.
\end{proof}


We have by definition, $F\wr \mathbb{Z}^k = \Sigma \rtimes_\alpha \mathbb{Z}^k$,  where  
$\Sigma$ denotes $\oplus_{x\in \mathbb{Z}^k} F_x$, and $\alpha(x)(g_y) =g_{x+y}$. Here $g_x$ 
denotes $g$ as an element of $F \cong F_x$.

We need a description of  automorphisms of a semidirect product $H \rtimes K$, where
$H$ is characteristic, as matrices $\begin{pmatrix}
a & b \\
0 & d
\end{pmatrix}$, where $a \in \operatorname{Aut}(H)$, 
$d \in \operatorname{Aut}(K)$, $b: K\to H$ satisfy 
\begin{itemize}
\item[(i)]  $b(k k')=b(k) \tau_{d(k)} (b(k'))$ for any $k, k'\in K$,
\item[(ii)] $a(\tau_k(h))=\tau_{b(k)d(k)}(a(h))$ for any $h\in H$, $k\in K$,
\end{itemize}
(see \cite[Theorem~1]{Curran2008}).
In contrast with the abelian case we can not reduce the consideration to the case $b=0$ (i.e. $b(k)=e_H$ for any $k\in K$).

Hence, since $\Sigma$ is characteristic in $\Gamma$ (as its torsion subgroup),  
we see, that 
an automorphism $\varphi$ can be defined by $a=\varphi':\Sigma \to \Sigma$, $d=\overline{\varphi}:
\mathbb{Z}\to \mathbb{Z}$ and $b: {\mathbb Z}^k \to \Sigma$ restricted to satisfy (in particular)
\begin{equation}\label{eq:maineq}
	\varphi'(\alpha(m)(h))=b(m)[\alpha(\overline{\varphi}(m))(\varphi'(h))] b(-m),\qquad h\in\Sigma,\quad m\in {\mathbb Z}^k.
\end{equation}
Then 
$
	(\varphi')^2(\alpha(m)(h))=\varphi'(b(m)) b(m)[\alpha(\overline{\varphi}^2(m))((\varphi')^2 (h))] b(-m)
\varphi'(b(-m)),
$
\begin{equation}\label{eq:maineq_deg}
	(\varphi')^q (\alpha(m)(h))=\tau_{\beta(q)} [\alpha(\overline{\varphi}^q(m))((\varphi')^q (h))]\mbox{ for }
\beta(q):=(\varphi')^{q-1}(b(m)) \cdots\varphi'(b(m)) b(m). 
\end{equation}

\begin{definition}
\rm For $\omega = \sum_m g_m \in \Sigma$,
denote the  \emph{support in ${\mathbb Z}^k$ of }$\omega$
by $\operatorname{supp}_{\mathbb{Z}^k}  (\omega)= \{m\in \mathbb{Z}^k \colon g_m \ne e\}$.
\end{definition}
 Since $\tau_g(g')=e$ if and only if $g'=e$, $g,g'\in F$, one has
\begin{equation}\label{eq:supports}
\operatorname{supp}_{\mathbb{Z}^k} (\tau_{\beta(q)} [\alpha(\overline{\varphi}^q(m))((\varphi')^q (h))])
=
\operatorname{supp}_{\mathbb{Z}^k} (\alpha(\overline{\varphi}^q(m))((\varphi')^q (h))).
\end{equation}

The following statement was discussed in \cite{gowon1,FraimanTroitsky} in the abelian case.

\begin{lemma}\label{lem:R_needed}
An authomorphism $\varphi: F\wr \mathbb{Z}^k \to F \wr \mathbb{Z}^k$, where $|F|<\infty$, has $R(\varphi)<\infty$ if and only if 
$R(\overline{\varphi})< \infty$ and $R(\tau_m \circ \varphi')<\infty$ for any $m \in  \mathbb{Z}^k$, 
where $\varphi': \oplus_m F_m \to \oplus_m F_m$ and $\overline{\varphi}: \mathbb{Z}^k
\to \mathbb{Z}^k$ are induced by $\varphi$
(in fact, it is sufficient to verify this for representatives of Reidemeister classes of $\overline{\varphi}$).
\end{lemma}

\begin{proof}
Suppose, $R(\varphi)<\infty$. By Lemma \ref{lem:extensions}, we have $R(\overline{\varphi})<\infty$. Then by Lemma  \ref{lem:Jab_fin}, we obtain
$|\mathbf{F}(\overline{\varphi})|<\infty$ (in fact, $|\mathbf{F}(\overline{\varphi})|=1$, because an automorphism of $\mathbb{Z}^k$ can not have finitely many fixed elements except of $0$). So, by Lemma \ref{lem:extensions}, $R(\varphi')<\infty$. Considering $\tau_z \circ \varphi$, which has $R(\tau_z \circ \varphi)=R(\varphi)<\infty$, instead of $\varphi$, we obtain in the same way that $R(\tau_z \circ \varphi')<\infty$.

 Conversely, having  $|\mathbf{F}(\overline{\varphi})|=1$, one can apply Lemma \ref{lem:extensions} (the formula from the last item) and use the equality $R(\tau_{\sigma s} \varphi')=R(\tau_{s} \varphi')$, where
$\sigma \in  \oplus_m F_m$, $s\in \mathbb{Z}^k$, see Lemma \ref{lem:shift}.
\end{proof}

To prove the main results, we need the following restriction
\begin{equation}\label{eq:condfinord}
\operatorname{order}(\varphi')=\operatorname{order}(\overline{\varphi})=:s.
\end{equation}
In many cases it is proved that this condition is fulfilled automatically \cite{TroLamp,troitskii2023reidemeister550080588}.

\begin{theorem}\label{teo:main_for_TBFT}
Suppose that $\varphi$ is an automorphism  of finite order of the restricted wreath product $G\wr {\mathbb Z}^k=\oplus_{m\in {\mathbb Z}^k} G_m \rtimes_\alpha {\mathbb Z}^k$, where $G$ is a finite group, with
(\ref{eq:condfinord}). If $R(\varphi')<\infty$,  then $R(\varphi')=1$.
\end{theorem}

\begin{proof} 
	Suppose, $R(\varphi')>1$. Then there exists an element $\sigma\in \Sigma$ such that
	$\sigma\not \in \{ e\}_{\varphi'}$. Hence $\sigma\not \in \{ e\}_{\varphi'_\sigma}$,
	where $\varphi'_\sigma$ is the restriction of $\varphi'$ onto the $\varphi'$-invariant subgroup $\Sigma_\sigma$ generated by $\sigma$. This follows from the evident observation
$$
 \{ e\}_{\varphi'_\sigma}=\{g (\varphi'_\sigma)^{-1} (g) \colon g \in \Sigma_\sigma\}
=\{g (\varphi')^{-1} (g) \colon g \in \Sigma_\sigma\} \subseteq \{g (\varphi')^{-1} (g) \colon g \in \Sigma \}= \{ e\}_{\varphi'}.
$$
In particular, $R(\varphi'_\sigma)>1$.   
Then by the definition, $\Sigma_\sigma$ is a finite group with generators $\sigma,\varphi'(\sigma),\dots,(\varphi')^s(\sigma)$  (because $\Sigma$ is locally finite). Hence, $\varphi'_\sigma$ has
	a nontrivial fixed element $\sigma_0$, $\varphi'_\sigma(\sigma_0)=\sigma_0$ and
	$\sigma_0\ne 0$ (see Lemma \ref{lem:fix_fin}).  Consider an element $m\in{\mathbb Z}^k$
from an orbit of maximal length $s$, i.e. $\overline{\varphi}^s(m)=m$ and 
$\overline{\varphi}^j(m)$ are pairwise distinct for $j=0,\dots,s-1$. 
Passing from $m$ to $n_1 m$,
	$n_1 \in {\mathbb Z}$, $m\in {\mathbb Z}^k$, if necessary, we can assume that the supports
	$\operatorname{supp}_{\mathbb{Z}^k}(\alpha(\overline{\varphi}^j(n_1 m))\sigma_0)$, $j=0,\dots,t$, do not intersect. Indeed, since the elements $\overline{\varphi}^j(m) \in \mathbb{Z}^k$ are distinct, $j=0,\dots,s-1$, and we can take a large $n_1$ such that
\begin{eqnarray*}
\|\overline{\varphi}^j(n_1 m) - \overline{\varphi}^l (n_1 m)\| &=&
\|n_1 \overline{\varphi}^j(m) - n_1 \overline{\varphi}^l (m)\| =\\
&=&|n_1| \cdot \|\overline{\varphi}^j(m) - \overline{\varphi}^l (m)\| > 2 \operatorname{diam}
(\operatorname{supp}_{\mathbb{Z}^k}( \sigma_0)).
\end{eqnarray*}
Then, by (\ref{eq:maineq_deg}) and (\ref{eq:supports}) the supports $\operatorname{supp}_{\mathbb{Z}^k}( (\varphi')^j(\alpha(n_1 m)\sigma_0)$ do not intersect, $j=0,\dots,s-1$. In particular, these elements commute
and $\sigma_1=\prod_{j=0}^{s-1} (\varphi')^j(\alpha(n_1 m)\sigma_0$ is a $\varphi'$-fixed element. Moreover,
$\sigma_1\ne e$, $\sigma_1\ne \sigma_0$.
Increasing $n=n_1,n_2,\dots$ ``in sufficiently large steps'' we obtain infinitely many distinct fixed elements in the same way.	Then by Lemma \ref{lem:Jab_fin_ord}, $R(\varphi')=\infty$. A contradiction.
\end{proof}	

\begin{theorem}\label{teo:main_TBFT}
Suppose that $\varphi$ is an automorphism  of finite order of the restricted wreath product $G\wr {\mathbb Z}^k=\oplus_{m\in {\mathbb Z}^k} G_m \rtimes_\alpha {\mathbb Z}^k$, where $G$ is a finite group and
(\ref{eq:condfinord}) is fulfilled. Then $\varphi$ has the TBFT$_f$ property. 
\end{theorem}

\begin{proof} 
By Lemma \ref{lem:R_needed}, $R(\varphi)<\infty$ implies
$R(\varphi')<\infty$. Then Theorem \ref{teo:main_for_TBFT} implies that $R(\varphi')=1$.
Considering $\tau_z \circ \varphi$ instead of $\varphi$ from the very beginning, we see that  $R(\tau_z \circ \varphi')=1$, for any $z\in {\mathbb Z}^k$.
Thus, by Lemma \ref{lem:extensions}, Reidemeister classes $\{g\}_\varphi$ of $\varphi$ are pull-backs of Reidemeister classes
$\{z\}_{\overline{\varphi}} $ of $\overline{\varphi}$ under the natural projection $\pi:G\wr {\mathbb Z}^k \to {\mathbb Z}^k$, i.e. $\{g\}_\varphi=\pi^{-1}(\{\pi(g)\}_{\overline{\varphi}})$.  So, if classes of $\overline{\varphi}$  are separated by an epimorphism $f: {\mathbb Z}^k \to A$ onto a finite 
abelian group $A$, 
then classes of $\varphi$ are separated by $f\circ\pi$. 
Thus by Lemma \ref{lem:tbft_separ}, the statement follows from TBFT$_f$ for abelian groups \cite{FelHill} (see also \cite{polyc}).
\end{proof}

Now we pass to a proof of conjecture (R) in the situation under consideration.

\begin{theorem}\label{teo:ja_wr}
Suppose that  a restricted wreath product $\Gamma=G\wr \mathbb{Z}^k$, where $G$ is an arbitrary finite group,
admits an automorphism $\varphi$ of finite order with $R(\varphi)<\infty$, for which
(\ref{eq:condfinord}) is fulfilled. Then
$\Gamma$ is solvable.
\end{theorem}

\begin{proof}
Similarly to the proof of Theorem \ref{teo:main_TBFT}, consider the $\varphi'$-invariant  subgroup $\Sigma_0\subset \Sigma$ generated by $G_0$. This is a finite subgroup generated by
$
G_0, \varphi' (G_0), \dots,  (\varphi')^{s-1} (G_0).
$ 
Then the restriction $\varphi'_0: \Sigma_0 \to \Sigma_0$ has no fixed elements as it is proved above (as the key step of the proof of Theorem \ref{teo:main_TBFT}).
Then by Theorem \ref{teo:Rowley95} the group $\Sigma_0$ is solvable. Then $G_0$ is the image of the composition
$
\Sigma_0 \hookrightarrow \Sigma \xrightarrow {p_0} G_0,
$
where $p_0$ is the natural projection. 
Hence $G_0$ is solvable. Then $\Sigma$ and $[\Gamma,\Gamma]\subseteq \Sigma$ are solvable. Thus, $\Gamma$ is solvable. 
\end{proof}

\begin{remark}
\rm An important observation related these groups is that typically
the subgroup $\Sigma$ does \emph{not} satisfy neither TBFT nor TBFT$_f$ as it was proved in \cite{FeTrZi20}
(of course with failure for \emph{another} automorphism then $\varphi'$).
\end{remark}

The author is indebted to Professor A.~A.~Klyachko for valuable criticism of the first version of the paper.

The work is supported by the Russian Science Foundation under grant 21-11-00080.

\def\cprime{$'$} \def\cprime{$'$} \def\cprime{$'$} \def\cprime{$'$}
  \def\cprime{$'$} \def\cprime{$'$} \def\cprime{$'$} \def\dbar{\leavevmode\hbox
  to 0pt{\hskip.2ex \accent"16\hss}d}
  \def\polhk#1{\setbox0=\hbox{#1}{\ooalign{\hidewidth
  \lower1.5ex\hbox{`}\hidewidth\crcr\unhbox0}}} \def\cprime{$'$}
  \def\cprime{$'$} \def\cprime{$'$}

\end{document}